%% file: BL16.tex
\documentclass[12pt]{article}

\usepackage{amsmath}
\usepackage{amssymb,amsbsy}
\usepackage{graphicx}
\usepackage{booktabs}
\usepackage{dsfont}
\usepackage{natbib}
\usepackage{pdfsync}
\usepackage{hyperref}
\usepackage{times}
\usepackage{bbm}
\usepackage{tikz}
\usepackage[lined,boxed,ruled,norelsize,algo2e]{algorithm2e}

\usepackage[small,bf]{caption}
\usepackage[top=1in,bottom=1in,left=1in,right=1in]{geometry}

\usetikzlibrary{arrows,intersections}
\usetikzlibrary{shapes.misc}

\tikzset{cross/.style={cross out, draw=black, minimum size=2*(#1-\pgflinewidth), inner sep=0pt, outer sep=0pt},
cross/.default={1pt}}

\include{Commands}

\newcommand{\fval}{\bold{fval}}
\newcommand{\LS}{\mathtt{line\_search}}
\newcommand{\diag}{\mathtt{diag}}
\newcommand{\sym}{\mathtt{sym}}
\newcommand{\tr}{\mathtt{tr}}
\newtheorem{definition}{Definition}
\newtheorem{proposition}{Proposition}
\newtheorem{theorem}{Theorem}
\newtheorem{lemma}{Lemma}
\newcommand{\BlackBox}{\rule{1.5ex}{1.5ex}}
\newenvironment{proof}{\par\noindent{\bf Proof\ }}{\hfill\BlackBox\\[2mm]}

\begin{document}

\title{Black-box optimization with a politician}

\author{S\'ebastien Bubeck\\
Microsoft Research\\
sebubeck@microsoft.com
\and
Yin Tat Lee\footnote{Most of this work were done while the author was at Microsoft Research, Redmond. The author was supported by NSF awards 0843915 and 1111109.}\\
MIT\\
yintat@mit.edu
}
\date{\today}

\maketitle

\begin{abstract} 
We propose a new framework for black-box convex optimization which is well-suited for situations where gradient computations are expensive. 
We derive a new method for this framework which leverages several concepts from convex optimization, from standard first-order methods (e.g. gradient descent or quasi-Newton methods) to analytical centers (i.e. minimizers of self-concordant barriers). We demonstrate empirically that our new technique compares favorably with state of the art algorithms (such as BFGS).
\end{abstract} 

\section{Introduction}
In standard black-box convex optimization \cite{NY83, Nes04, Bub15} first-order methods interact with an {\em oracle}: given a query point $x$, the oracle reports the value and gradient of the underlying objective function $f$ at $x$. In this paper we propose to replace the oracle by a {\em politician}. Instead of answering the original query $x$ the politician changes the question and answers a new query $y$ which is guaranteed to be better than the original query $x$ in the sense that $f(y) \leq f(x)$. The newly selected query $y$ also depends on the history of queries that were made to the politician. Formally we introduce the following definition (for sake of simplicty we write $\nabla f(x)$ for either a gradient or a subgradient of $f$ at $x$).

\begin{definition}
Let $\cX \subset \R^n$ and $f : \mathcal{X} \rightarrow \R$. A politician $\Phi$ for $f$ is a mapping from $ \mathcal{X} \times \cup_{k = 0}^{\infty} (\mathcal{X} \times \R \times \R^n)^k$ to $\mathcal{X}$ such that for any $k \geq 0, x \in \mathcal{X}, h \in (\mathcal{X} \times \R \times \R^n)^k$ one has $f(\Phi(x,h)) \leq f(x)$. Furthermore when queried at $x$ with history $h$ a politician for $f$ also output $f(\Phi(x,h))$ and $\nabla f(\Phi(x,h))$ (in order to not overload notation we do not include these outputs in the range of $\Phi$). 
\end{definition}

Let us clarify the interaction of a first-order method with a politician. Note that we refer to the couple (first-order method, politician) as the {\em algorithm}. Let $M : \cup_{k=0}^{\infty} (\cX \times \R \times \R^n)^k \rightarrow \cX$ be a first-order method and $\Phi$ a politician for some function $f: \cX \rightarrow \R$. The course of the algorithm $(M, \Phi)$ then goes as follows: at iteration $k+1$ one first calculates the method's query point $x_{k+1} = M(h_k)$ (with $h_0 = \emptyset$), then one calculates the politician's new query point $y_{k+1}=\Phi(x_{k+1}, h_k)$ and the first order information at this point $(f(y_{k+1}), \nabla f(y_{k+1}))$, and finally one updates the history with this new information $h_{k+1} = (h_k, (y_{k+1}, f(y_{k+1}), \nabla f(y_{k+1})))$. Note that a standard oracle simply corresponds to a politician $\mathcal{O}$ for $f$ such that $\mathcal{O}(x,h) = (x, f(x), \nabla f(x))$ (in particular the algorithm $(M, \mathcal{O})$ is the usual algorithm corresponding to the first-order method $M$).

The philosophy of the above definition is that it gives in some sense an automatic way to combine different optimization algorithms. Say for example that we wish to combine the ellipsoid method with gradient descent. One way to do so is to design an ``ellipsoidal politician": the politician keeps track of a feasible ellipsoidal region based on the previously computed gradients, and when asked with the query $x$ the politician chooses as a new query $y$ the result of a line-search on the line between $x$ and the center of current ellipsoid. Gradient descent with this ellipsoidal politician would then replace the step $x \leftarrow x - \eta \nabla f(x)$ by $x \leftarrow y - \eta \nabla f(y)$. The hope is that in practice such a combination would integrate the fast incremental progress of gradient descent with the geometrical progress of the ellipsoid method.

In this paper we focus on unconstrained convex optimization. We are particularly interested in situations where calculating a (sub)gradient has superlinear complexity (i.e., $\gg n$) such as in logistic regression and semidefinite programming. In such cases it is natural to try to make the most out of the computed gradients by incorporating geometric reasoning (such as in the ellipsoid method). We do so by introducing the {\em geometric politician} (Section \ref{sec:geopol}), which is based on a combination of the recent ideas of \cite{BLS15} with standard cutting plane/interior point methods machinery (through the notion of a ``center" of a set, see Section \ref{sec:center}). For a given first order method $M$, we denote by $M+$ the algorithm obtained by running $M$ with the geometric politician. We demonstrate empirically (Section \ref{sec:exp}) the effectiveness of the geometric politician on various standard first-order methods for convex optimization (gradient descent, Nesterov's accelerated gradient descent, non-linear conjugate gradient, BFGS). In particular we show that BFGS+ is a surprisingly robust and parameter-free algorithm with state of the art performance across a wide range of problems (both smooth and non-smooth).

\section{Affine invariant politician} \label{sec:affpol}
As mentioned above we assume that the complexity of computing the map $x \mapsto \nabla f(x)$ is superlinear. This implies that we can afford to have a politician such that the complexity of computing the map $(x, h) \mapsto \Phi(x,h)$ is $O(n \times \mathrm{poly}(k))$ (we think of the number of iterations $k$ as typically much smaller than the dimension $n$).
We show in this section that this condition is (essentially) automatically satisfied as long as the politician is affine invariant in the following sense (we use a slight abuse of language and refer to a map $f \mapsto \Phi_f$, where $\Phi_f$ is a politician for $f$, as a politician):

\begin{definition}
A politician $f \mapsto \Phi_f$ is called affine invariant if for any function $f$ and any affine map $T : \R^m \rightarrow \R^n$ such that $T(x) = z + Lx$ for some matrix $L$, $k \geq 0, x \in \R^n, (y_i, v_i, g_i) \in \R^m \times \R \times \R^n$, one has 
$$ T(\Phi_{f \circ T}(x,(y_i, v_i, L^{\top} g_i)_{i \in [k]})) = \Phi_f(T(x), (T(y_i), v_i, g_i)_{i \in [k]}) .$$
We say that an affine invariant politician has cost $\psi : \N \rightarrow \N$ if for any $f: \R^k \rightarrow \R$ the map $(x,h) \in \R^k \times (\R^k \times \R \times \R^k)^k \mapsto \Phi_f(x,h)$ can be computed in time $\psi(k)$.
\end{definition}

\begin{proposition} \label{prop:1}
Let $\Phi$ be an affine invariant politician with cost $\psi$. Then for any $f: \R^n \rightarrow \R$, $(y_i, v_i, g_i) \in \R^n \times \R \times \R^n, i \in [k]$ and $x, y_i \in y_1 + \mathrm{Span}(g_1, \hdots, g_k)$ one can compute $\Phi_f(x,(y_i, v_i, g_i)_{i \in[k]}) \in \R^n$ in time $\psi(k) + O(n k^2)$.
\end{proposition}

\begin{proof}
Let $G$ be the $n \times k$ matrix with $i^{th}$ column given by $g_i$. We consider the $QR$ decomposition of $G$ which can be computed in time $O(n k^2)$, that is $Q$ is an $n \times k$ matrix and $R$ a $k\times k$ matrix such that $G = QR$ and $Q^{\top} Q = \mathrm{I}_k$. Let $T$ be the affine map defined by $T=y_1 + Q$. Note that since $x \in y_1 + \mathrm{Span}(g_1, \hdots, g_k)$ one has $x = T(Q^{\top}(x-y_1))$ (and similarly for $y_i$). Thus by affine invariance one has
\begin{align*}
\Phi_f(x, (y_i, v_i, g_i)) & = \Phi_f(T(Q^{\top}(x-y_1)), (T(Q^{\top}(y_i-y_1)), v_i, g_i)) \\
& = y_1 + Q \Phi_{f \circ T}(Q^{\top} (x-y_1), (Q^{\top} (y_i-y_1), v_i, R_i)) ,
\end{align*}
where $R_i$ is the $i^{th}$ column of $R$.
Furthermore by definition of the cost $\psi$ and since $f \circ T$ is defined on $\R^k$ we see that this last quantity can be computed in time $\psi(k) + O(n k^2)$, thus concluding the proof.
\end{proof}

The above proposition shows that with an affine invariant politician and a first order method $M$ verifying for any $(y_i, v_i, g_i)_{i \in [k]} \in (\R^n \times \R \times \R^n)^k$,
$$M((y_i, v_i, g_i)_{i \in [k]}) \in y_1 + \mathrm{Span}(y_1, \hdots, y_k, g_1, \hdots, g_k) ,$$
one can run $k$ steps of the corresponding algorithm in time $O(n k^2 + k \psi(k))$ plus the time to compute the $k$ function values and gradients of the underlying function $f$ to be optimized. Note that one gets a time of $O(n k^2)$ instead of $O(n k^3)$ as one can store the $QR$ decomposition from one step to the next, and updating the decomposition only cost $O(n k)$.

\section{Geometric politician} \label{sec:geopol}
We describe in this section the {\em geometric politician} which is based on ideas developed in \cite{BLS15}. A key observation in the latter paper is that if $f$ is a $\alpha$-strongly convex function minimized at $x^*$ then one has for any $x$,
\[
\left\|x^{*}-x-\frac{1}{\alpha}\nabla f(x)\right\|^{2}\leq\frac{\|{\nabla f(x)}\|^{2}}{\alpha^{2}}-\frac{2}{\alpha}\left(f(x)-f(x^{*})\right).
\]
This motivates the following definition:
$$\mB(x,\alpha,\fval) :=  \left\{ z \in \R^n : \left\|{z-x-\frac{1}{\alpha}\nabla f(x)}\right\|^{2} \leq \frac{\|{\nabla f(x)}\|^{2}}{\alpha^{2}}-\frac{2}{\alpha}\left(f(x)-\fval\right)\right\} .$$
In particular given the first order information at $y_1, \hdots, y_k$ one knows that the optimum $x^*$ lies in the region $R_k \subset \R^n$ defined by
\begin{equation}
R_k=\bigcap_{i\in[k]}B(y_i,\alpha,\fval)\text{ where }\fval=\min_{i\in[k]}f(y_i).\label{eq:feasible_region}
\end{equation}
Now suppose that given this first order information at $y_1,\hdots, y_k$ the first order method asks to query $x$. How should we modify this query in order to take into account the geometric information that $x^* \in R_k$? First observe that for any $z$, $\mB(z,\alpha, \fval)$ is contained in a halfspace that has $z$ on its boundary (in the limiting case $\alpha \rightarrow 0$ the set $\mB(z,\alpha, f(z))$ is exactly a halfspace). In particular if the next query point $y_{k+1}$ is the center of gravity of $R_k$ then we have that the volume of $R_{k+1}$ is at most $1-1/e$ times the volume of $R_k$ (see \cite{Gru60}), thus leading to an exponential convergence rate. However the region $R_k$ can be very large initially, and the center of gravity might have a large function value and gradient, which means that $R_k$ would be intersected with a large sphere (possibly so large that it is close to a halfspace). On the other hand the first order method recommends to query $x$, which we can think of as a local improvement of $y_k$, which should lead to a much smaller sphere. The issue is that the position of this sphere might be such that the intersection with $R_k$ is almost as large as the sphere itself. In order to balance between the geometric and function value/gradient considerations we propose for the new query to do a line search between the center of $R_k$ and the recommended query $x$. The geometric politician follows this recipe with two important modifications: (i) there are many choices of centers that would guarantee an exponential convergence rate  while being much easier to compute than the center of gravity, and we choose here to consider the {\em volumetric center}, see Section \ref{sec:center} for the definition and more details about this notion; (ii) we use a simple heuristic to adapt online the strong convexity parameter $\alpha$, namely we start with some large value for $\alpha$ and if it happens that the feasible region $R_k$ is empty then we know that $\alpha$ was too large, in which case we reduce it. We can now describe formally the geometric politician, see Algorithm \ref{alg:geopoli}. Importantly one can verify that the geometric politician is affine invariant and thus can be implemented efficiently (see the proof of Proposition \ref{prop:1}).

\begin{algorithm2e}
\caption{Geometric Politician}
\label{alg:geopoli}

\SetAlgoLined

\textbf{Parameter: }An upper bound on the strong convexity parameter $\alpha$.
(Can be $+\infty$.)

\textbf{Input: } Query $x$, past queries and the corresponding first order information $(y_i, f(y_i), \nabla f(y_i))_{i\in[k]}$.

Let $\fval=\min_{i\in[k]}f(y_i)$ and the feasible region $R_k(\alpha)=\bigcap_{i\in[k]}\mB(y_i,\alpha,\fval).$

\If{$R_k(\alpha)=\emptyset$}{

Let $\alpha$ be the largest number such that $R_k(\alpha)\neq\emptyset$.

$\alpha\leftarrow\alpha/4$.

}

Let $y_{k+1}=\argmin_{y \in \{t x + (1-t) c(R_k(\alpha)), t \in \R\}} f(y)$

where $c(R_k(\alpha))$
is the volumetric center of $R_k(\alpha)$ (see Section \ref{sec:center}).

\textbf{Output:} $y_{k+1}$, $f(y_{k+1})$ and $\nabla f(y_{k+1})$.

\end{algorithm2e}

\section{Volumetric center} \label{sec:center}
The volumetric barrier for a polytope was introduced in \cite{Vai96} to construct an algorithm with both the oracle complexity of the center of gravity method and the computational complexity of the ellipsoid method (see [Section 2.3, \cite{Bub15}] for more details and \cite{LSW15} for recent advances on this construction). Recalling that the standard logarithmic barrier $F_P$ for the polytope $P=\{x \in \mathbb{R}^n : a_i^{\top} x < b_i, i \in [m]\}$ is defined by
$$F_P(x) = - \sum_{i=1}^m \log(b_i - a_i^{\top} x) ,$$
one defines the volumetric barrier $v_P$ for $P$ by
$$v_P(x) = \mathrm{logdet}(\nabla^2 F_P(x) ) .$$
The volumetric center $c(P)$ is then defined as the minimizer of $v_P$. In the context of the geometric politician (see Algorithm \ref{alg:geopoli}) we are dealing with an intersection of balls rather than an intersection of halfspaces. More precisely the region of interest is of the form:
\[
R = \bigcap_{i=1}^{k}\left\{ x\in\R^n:\|x-c_{i}\|\leq r_{i}\right\} .
\]
For such a domain the natural self-concordant barrier to consider is:
\[
F_R(x) = -\frac{1}{2}\sum_{i=1}^{k}\log\left(r_{i}^{2}-\|x-c_{i}\|^{2}\right).
\]
The volumetric barrier is defined as before by
\[
v_R(x) = \mathrm{logdet} ( \nabla^{2} F_R(x) ) ,
\]
and the volumetric center of $R$ is the minimizer of $v_R$. It is shown in \cite{Ans04} that $v_R$ is a self-concordant barrier which means that the center can be updated 
(when a new ball is added to $R$) via few iterations of Newton's method. Often in practice, it takes less than 5 iterations to update the minimizer of a self-concordant barrier \cite{GV99, BDGV95} when we add a new constraint. Hence, the complexity merely depends on how fast we can compute the gradient and Hessian of $F_R$ and $v_R$.

\begin{proposition}
For the analytic barrier $F_R$, we have that 
\begin{align*}
\nabla F_{R}(x)  = &  A^{\top}1_{k\times1},\\
\nabla^{2}F_{R}(x)  = &  2A^{\top}A+\lambda^{(1)}I
\end{align*}
where $d$ is a vector defined by $\left(r_{i}^{2}-\|x-c_{i}\|^{2}\right)^{-1}$,
$A$ is a $k\times n$ matrix with $i^{th}$ row given by $d_{i}(x)(x-c_{i})$,
$\lambda^{(p)}=\sum_{i\in[k]}d_{i}^{p}(x)$ and $1_{k\times1}$
is a $k\times1$ matrix with all entries being $1$.

For the volumetric center, we have that
\begin{align*}
\nabla v_{R}(x) & = \left(\left(2\tr H^{-1}\right)\mathrm{I}+4H^{-1}\right)A^{\top}d+8A^{\top}\sigma, \\
\nabla^{2}v_{R}(x)  & = 48A^{\top}\Sigma A-64A^{\top}\left(A H^{-1}A^{\top}\right)^{(2)}A \\
  & + \left(8 \tr(D \Sigma)+2\lambda^{(2)}\tr(H^{-1})\right)I +4\lambda^{(2)}H^{-1}\\
   & +8 \tr(H^{-1})A^{\top}D A+16\sym\left(A^{\top}D A H^{-1}\right)\\
   & -4 \tr(H^{-2})A^{\top}D \mathrm{J} D A-8H^{-1}A^{\top}D \mathrm{J} D A H^{-1}\\
   & -8\sym(A^{\top}D \mathrm{J} D A H^{-2})-8\left(d^{\top}A H^{-1}A^{\top}d\right)H^{-1}\\
   & -16\sym\left(A^{\top}\diag\left(A H^{-2}A^{\top}\right) \mathrm{J} D A\right)\\
   & -32\sym\left(A^{\top}\diag(A H^{-1}A^{\top}d)A H^{-1}\right)
\end{align*}
where $H=\nabla^{2}F_{R}(x)$, $\sigma_{i}=e_{i}^{\top}A H^{-1}A^{\top}e_{i}$,
$e_{i}$ is the indicator vector with $i^{th}$ coordinate, $\mathrm{J}$
is a $k\times k$ matrix with all entries being 1, $\sym(B)=B+B^{\top}$,
$\diag(v)$ is a diagonal matrix with $\diag(v)_{ii}=v_{i}$, $\Sigma=\diag(\sigma)$,
and $B^{(2)}$ is the Schur square of $B$ defined by $B_{ij}^{(2)}=B_{ij}^{2}$.
\end{proposition}

The above proposition shows that one step of Newton method for analytic center requires 1 dense matrix multiplication and solving 1 linear system; and for volumetric center, it requires 5 dense matrix multiplications, 1 matrix inversion and solving 1 linear system if implemented correctly. Although the analytic center is a more popular choice for ``geometrical'' algorithms, we choose volumetric center here because it gives a better convergence rate \cite{Vai96,AV95} and the extra cost $\psi(k)$ is negligible to the cost of updating QR decomposition $nk$.

\section{Experiments} \label{sec:exp}

\input{experiments}

\section{Discussion} \label{sec:disc}
First order methods generally involve only very basic operations at each step (addition, scalar multiplication). In this paper we formalize each step's operations (besides the gradient calculation) as the work of the politician. We showed that the cost per step of an affine invariant politician $\psi(k)$ is negligible compared to the gradient calculation (which is $\Omega(n)$). This opens up a lot of possibilities: instead of basic addition or scalar multiplication one can imagine computing a center of gravity, solving a linear program, or even searching over an exponential space (indeed, say $k<30$ and $n>10^{10}$, then $2^k < n$). Our experiments demonstrate the effectiveness of this strategy. On the other hand from a theoretical point of view a lot remains to be done. For example one can prove results of the following flavor:

\begin{theorem}
Let $f$ such that $\alpha \mathrm{I} \preceq \nabla^2 f(x) \preceq \beta \mathrm{I}, \forall x \in \R^n$ and let $\kappa = \beta / \alpha$. Suppose that in the Geometric Politician we replace the volumetric center by the center of gravity or the center of the John ellipsoid. Let $y_k$ be the output of the $k^{th}$ step of SD$+$ with some initial point $x_0$. Then, we have that 
$$f(y_k) - f(x^*) \leq \kappa \left (1-\frac{1}{\Theta(\min(n \log(\kappa), \kappa))} \right)^k \left(f(x_0) - f(x^*) \right ).$$
and
$$ f(y_k) - f(x^*) \leq \frac{2 \beta R^2}{k+4}$$
where $R = \max_{f(x)\leq f(x_0)} \| x - x^* \|$.
\end{theorem}

This claim says that, up to a logarithmic factor, SD$+$ enjoys simultaneously the incremental progress of gradient descent and the geometrical progress of cutting plane methods. There are three caveats in this claim:
\begin{itemize}
\item We use the center of gravity or the center of the John ellipsoid instead of the volumetric center. 
Note however that it is well-known that the volumetric center is usually more difficult to analyze, \cite{Vai96,AV95}.
\item The extraneous $\log(\kappa)$ comes from the number of potential restart when we decrease $\alpha$.
Is there a better way to learn $\alpha$ that would not incur this additional logarithmic term?
\item \cite{BLS15} shows essentially that one can combine the ellipsoid method with gradient descent to achieve the optimal $1-\sqrt{1/\kappa}$ rate. Can we prove such a result for SD$+$?
\end{itemize}

The geometric politician could be refined in many ways. Here are two simple questions that we leave for future work:
\begin{itemize}
\item One can think that gradient descent stores 1 gradient information, accelerated gradient descent stores 2 gradient information, and our method stores all past gradient information. We believe that neither 1, 2 nor all is the correct answer. Instead, the algorithm should dynamically decide the number of gradients to store based on the size of its memory, the cost of computing gradients, and the information each gradient reveals. 
\item Is there a stochastic version of our algorithm? How well would such a method compare with state of the art stochastic algorithms such as SAG \cite{LRSB12} and SVRG \cite{JZ13}?
\end{itemize}

\appendix 

\section{Convergence of SD$+$}
Let $f$ such that $\alpha\mathrm{I}\preceq\nabla^{2}f(x)\preceq\beta\mathrm{I}$
for all $x\in\mathbb{R}^{n}$ and let $\kappa=\beta/\alpha$. Let
$y_{k}$ be the output of the $k^{th}$ step of SD$+$ (where the volumetric center is replaced by the center of gravity or the center of the John ellipsoid) with some initial
point $x_{0}$. We prove two rates of convergence for SD$+$, one with the condition number $\kappa$, and one with the ambient dimension $n$. We start by the former.

\begin{theorem}
\label{th:kappa_convergence} One has
\[
f(y_{k})-f(x^{*})\leq\left(1-\frac{1}{\kappa}\right)^{k}\left(f(x_{0})-f(x^{*})\right)
\]
and
\[
f(y_{k})-f(x^{*})\leq\frac{2\beta r^{2}}{k+4}
\]
where $r=\max_{f(x)\leq f(x_{0})}\|x-x^{*}\|.$\end{theorem}
\begin{proof}
Let $\delta_{k}=f(y_{k})-f(x^*)$. Since $f$ is $\alpha$-strongly convex
we have that
\[
\delta_{k}\leq\frac{1}{2\alpha}\|\nabla f(y_{k})\|^{2}.
\]
Due to the decrease guarantee of politicians and the line search in steepest descent, we have that $f(y_{k+1})\leq f(x_{k+1})\leq f(y_{k})-\frac{1}{2\beta}\|\nabla f(y_{k}) \|^{2}$
and hence 
\begin{equation} \label{eq:forlater}
\delta_{k}-\delta_{k+1}\geq\frac{1}{2\beta}\|\nabla f(y_{k})\|^{2}\geq\frac{\delta_{k}}{\kappa}.
\end{equation}
Hence, we have $\delta_{k+1}\leq\left(1-\frac{1}{\kappa}\right)\delta_{k}$
and this gives the first inequality.

To obtain a rate independent of $\alpha$ we instead use the following estimate
\[
\delta_{k}\leq\left\langle \nabla f(y_{k}),y_{k}-x^{*}\right\rangle \leq\|\nabla f(y_{k})\| \cdot \|y_{k}-x^{*}\|.
\]
Using the decrease guarantee of politicians and line search we have
that $f(y_{k})\leq f(x_{k})\leq f(y_{k-1})\leq\cdots\leq f(x_{0})$,
and thus by definition of $R$:
\[
\|y_{k}-x^{*}\|\leq r.
\]
Due to the line search in steepest descent again, we have that 
\[
\delta_{k}-\delta_{k+1}\geq\frac{1}{2\beta}\|\nabla f(y_{k})\|^{2}\geq\frac{1}{2\beta}\left(\frac{\delta_{k}}{r}\right)^{2}.
\]
Since 
$\delta_{k}\geq\delta_{k+1}$, we have
\[
\frac{1}{\delta_{k+1}}-\frac{1}{\delta_{k}}=\frac{\delta_{k}-\delta_{k+1}}{\delta_{k}\delta_{k+1}}\geq\frac{\delta_{k}-\delta_{k+1}}{\delta_{k}^{2}}\geq\frac{1}{2\beta r^{2}}.
\]
So, by induction, we have that $\frac{1}{\delta_{k}}\geq\frac{1}{\delta_{0}}+\frac{k}{2\beta r^{2}}$.
Now, we note that 
\[
\delta_{0}\leq\left\langle \nabla f(x^{*}),x_{0}-x^{*}\right\rangle +\frac{\beta}{2}\|x_{0}-x^{*}\|^{2}\leq\frac{\beta r^{2}}{2}.
\]
Thus, we have that
\[
\delta_{k}\leq\frac{2\beta r^{2}}{k+4}.
\]

\end{proof}
We now turn to the dimension dependent analysis of SD$+$. We first show a simple geometric result, namely that if an intersection of spheres has a ``small'' volume then the intersection must lie close 
close to the boundary of one of the spheres.
\begin{lemma}
\label{lem:width_R}Let $R=\cap_{i=1}^{k}\{x\in\mathbb{R}^{n}:\|x-c_{i}\|\leq r_{i}\}$, $D=\max_{i \in [k]}r_{i}$, and $\omega_{n}$ the volume of the unit ball in $\R^n$.
Then, there exists $i \in [k]$ such that for all $x\in R$,
\[
\|x-c_{i}\|^{2}\geq r_{i}^{2}-24k^{2}\left(\frac{\vol R}{D^{n}\omega_{n}}\right)^{1/n}D^{2} .
\]
\end{lemma}
\begin{proof}
Since $-\log(1-\|x\|^2)$ is a $1$-self concordant barrier function,
$2F_{R}$ is a $k$-self concordant function. Let $y$ be the minimizer
of $F_{R}$. Let $E=\{x\in\mathbb{R}^{n}:x^{\top}\nabla^{2}(2F_{R})(y)x\leq1\}$.
Theorem 4.2.6 in \cite{Nes04} shows that 
\begin{equation}
y+E\subset R\subset y+(k+2\sqrt{k})E.\label{eq:rounding_FR}
\end{equation}
In particular, this shows that $\vol E\leq\vol R$. We have that
\[
\left(\det\nabla^{2}F_{R}(y)\right)^{1/2}=\frac{\omega_{n}}{2^{n/2}}\frac{1}{\vol E}\geq\frac{\omega_{n}}{2^{n/2}}\frac{1}{\vol R}.
\]
By the AM-GM inequality, we have that
\begin{equation}
\frac{\tr\nabla^{2}F_{R}(y)}{n}\geq\frac{1}{2}\left(\frac{\omega_{n}}{\vol R}\right)^{2/n}.\label{eq:tr_vol_ineq}
\end{equation}
By Proposition \ref{prop:1}, we have that $\nabla^{2}F_{R}(y)=2A^{\top}A+\lambda^{(1)}\mathrm{I}$
and hence,
\begin{eqnarray*}
\tr\nabla^{2}F_{R}(y) & = & 2\tr A^{\top}A+n\lambda^{(1)}\\
 & = & 2\sum_{i=1}^{k}\frac{\|y-c_{i}\|^{2}}{(r_{i}^{2}-\|y-c_{i}\|^{2})^{2}}+n\sum_{i=1}^{k}\frac{1}{r_{i}^{2}-\|y-c_{i}\|^{2}}.
\end{eqnarray*}
Applying (\ref{eq:tr_vol_ineq}), we have that
\[
\frac{2}{n}\sum_{i=1}^{k}\frac{\|y-c_{i}\|^{2}}{(r_{i}^{2}-\|y-c_{i}\|^{2})^{2}}+\sum_{i=1}^{k}\frac{1}{r_{i}^{2}-\|y-c_{i}\|^{2}}\geq\frac{1}{2}\left(\frac{\omega_{n}}{\vol R}\right)^{2/n}
\]
So, there exists $i$ such that
\[
\frac{\|y-c_{i}\|^{2}}{(r_{i}^{2}-\|y-c_{i}\|^{2})^{2}}\geq\frac{n}{8k}\left(\frac{\omega_{n}}{\vol R}\right)^{2/n}\text{ or }\frac{1}{r_{i}^{2}-\|y-c_{i}\|^{2}}\geq\frac{1}{4k}\left(\frac{\omega_{n}}{\vol R}\right)^{2/n}.
\]
Using $\|y-c_{i}\|\leq r_{i}\leq D$ and $\vol R\leq D^{n}\omega_{n}$,
we have that
\begin{eqnarray*}
r_{i}^{2}-\|y-c_{i}\|^{2} & \leq & \max\left(\sqrt{\frac{8k}{n}}\left(\frac{\vol R}{\omega_{n}}\right)^{1/n}\|y-c_{i}\|,4k\left(\frac{\vol R}{\omega_{n}}\right)^{2/n}\right)\\
 & \leq & 4k\left(\frac{\vol R}{D^{n}\omega_{n}}\right)^{1/n}D^{2}.
\end{eqnarray*}
Therefore, the width of the ellipsoid $E$ in the direction $y-c_{i}$ is at most 
\[
r_i - \|y-c_i\| \leq
r_{i}-\sqrt{r_{i}^{2}-4k\left(\frac{\vol R}{D^{n}\omega_{n}}\right)^{1/n}D^{2}}.
\]
The right hand side of (\ref{eq:rounding_FR}) shows that, for all
$x\in R$, we have
\begin{eqnarray*}
\|x-c_{i}\| & \geq & r_{i}-(1+k+2\sqrt{k})\left(r_{i}-\sqrt{r_{i}^{2}-4k\left(\frac{\vol R}{D^{n}\omega_{n}}\right)^{1/n}D^{2}}\right)\\
 & \geq & r_{i}-(1+k+2\sqrt{k})\left(r_{i}-r_{i}\left(1-4k\left(\frac{\vol R}{D^{n}\omega_{n}}\right)^{1/n}\frac{D^{2}}{r_{i}^{2}}\right)\right)\\
 & \geq & \left[1-12k^{2}\left(\frac{\vol R}{D^{n}\omega_{n}}\right)^{1/n}\frac{D^{2}}{r_{i}^{2}}\right]r_{i}.
\end{eqnarray*}
Hence, we have
\[
\|x-c_{i}\|^{2}\geq\left[1-24k^{2}\left(\frac{\vol R}{D^{n}\omega_{n}}\right)^{1/n}\frac{D^{2}}{r_{i}^{2}}\right]r_{i}^{2}.
\]

\end{proof}

Finally, equipped with the above geometrical result, we can bound the convergence of SD$+$ using the dimension $n$. We start with a lemma taking care of the adaptivity to the strong convexity in the geometric politician.
\begin{lemma}
\label{lem:n_convergence}  
In the first $k=\Theta(n\log(\frac{\kappa n}{\varepsilon}))$ iterations, either
SD$+$ restarts the estimate of the strong convexity or
\[
f(y_k)-f(x^{*})\leq \varepsilon \left(f(x_{0})-f(x^{*})\right).
\]
\end{lemma}
\begin{proof}
The decrease guarantee and the smoothness imply that
\[
\frac{\|\nabla f(y_{k})\|^{2}}{2\beta}\leq f(y_{k})-f(x^{*})\leq f(x_{0})-f(x^{*}).
\]
Therefore, all the spheres found by the geometric politician have radius
squared 
at most
$D^{2}$ where, denoting $\overline{\alpha}$ for the convexity upper bound the algorithm
is currently using,
\[
D^{2}= \max_{k \geq 1} \frac{\|\nabla f(y_{k})\|^{2}}{\overline{\alpha}^{2}}\leq\frac{2\beta(f(x_{0})-f(x^{*}))}{\overline{\alpha}^{2}} .
\]
Lemma \ref{lem:width_R} shows that for any step $k$,
there is $i \in [k]$ such that for all $x\in R_{k}$,
\[
\|x-c_{i}\|^{2}\geq r_{i}^{2}-\frac{48\beta k^{2}}{\overline{\alpha}^{2}}\left(\frac{\vol R_{k}}{D^{n}\omega_{n}}\right)^{1/n}(f(x_{0})-f(x^{*})) .
\]
Let $k=\Theta(n\log(\frac{\kappa n}{\varepsilon}))$ and recall the discussion in Section \ref{sec:geopol} about the volume decrease of the geometric politician with the center of gravity (the same discussion applies to the John ellipsoid). We see that if the algorithm does not restart $\overline{\alpha}$ within the first $k$ iterations then we have
\[
\frac{\vol R_{k}}{D^{n}\omega_{n}} = \left(O\left(\frac{\varepsilon}{\kappa^2 k^2}\right) \right)^{n} ,
\]
and hence (for an appropriate numerical constant in $k$)
\begin{equation}
\|x-c_{i}\|^{2}\geq r_{i}^{2}-\frac{\varepsilon(f(x_{0})-f(x^{*}))}{\overline{\alpha}\kappa}.\label{eq:width_R}
\end{equation}

Recall from \eqref{eq:forlater} that
\[
f(y_{k+1})\leq f(y_{k})-\frac{f(y_{k})-f(x^{*})}{\kappa} ,
\]
and therefore we have (by the improvement of the previous balls):
\[
R_{k+1}\subset\left\{ \|x-c_{i}\|^{2}\leq r_{i}^{2}-\frac{2(f(y_{k})-f(x^{*}))}{\overline{ \alpha} \kappa}\right\} \cap R_{k}.
\]
However, from \eqref{eq:width_R}, we know that either the above intersection
is empty or $f(y_{k})-f(x^{*})< \varepsilon (f(x_{0})-f(x^{*}))$. This
proves the statement.\end{proof}
\begin{theorem}
\label{th:n_convergence_2} We have that
\[
f(y_{k})-f(x^{*})\leq\kappa \left(1-\frac{1}{\Theta(n\log(\kappa))}\right)^{k}\left(f(x_{0})-f(x^{*})\right).
\]
\end{theorem}
\begin{proof}
If $\kappa<n$, the statements follows from Theorem \ref{th:kappa_convergence}.
Hence, we can assume $\kappa\geq n$.

Set $T=\Theta(n\log(\frac{n\kappa}{\varepsilon})\log(\kappa))$, Lemma
\ref{lem:n_convergence} shows that for every $\Theta(n\log(\frac{n\kappa}{\varepsilon}))$
iteration, the algorithm either finds $y$ such that 
\[
f(y)-f(x^{*})\leq \varepsilon \left(f(x_{0})-f(x^{*}) \right)
\]
or decreases $\alpha_{k}$ by a constant where $\alpha_{k}$ is the
convexity upper bound the algorithm is using at $k^{th}$ iteration.
Note that $\alpha_{1}\leq\beta$ because of the line search, and thus the algorithm
can restart $\alpha_{k}$ at most $\log(\kappa)$ many times. Hence,
after $T$ iterations, we must have
\[
f(y_{T})-f(x^{*})\leq \varepsilon\left(f(x_{0})-f(x^{*})\right) ,
\]
thus concluding the proof.
\end{proof}

\bibliographystyle{plainnat}
\bibliography{newbib}

\end{document}

%% file: Commands.tex
\newcommand{\vol}{\mathrm{vol}}

\renewcommand{\phi}{\varphi}

\newcommand{\mB}{\mathrm{B}}

\newcommand{\N}{\mathbb{N}}
\newcommand{\R}{\mathbb{R}}

\newcommand{\cX}{\mathcal{X}}

\def\ds1{\mathds{1}}
\renewcommand{\epsilon}{\varepsilon}

\newcommand{\argmin}{\mathop{\mathrm{argmin}}}

\newlength{\minipagewidth}
\newcommand{\beq}{\begin{equation}}
\newcommand{\eeq}{\end{equation}}

\newcommand{\beqa}{\begin{eqnarray}}
\newcommand{\eeqa}{\end{eqnarray}}

\newcommand{\beqan}{\begin{eqnarray*}}
\newcommand{\eeqan}{\end{eqnarray*}}

\def\ba#1\ea{\begin{align*}#1\end{align*}} 
\def\banum#1\eanum{\begin{align}#1\end{align}} 

%% file: experiments.tex
\setcounter{topnumber}{1}

In this section, we compare the geometric politician against two libraries
for first order methods, minFunc \cite{schmidt2012minfunc} and TFOCS \cite{becker2011templates}. Both are popular MATLAB
libraries for minimizing general smooth convex functions. Since the
focus of this paper is all about how to find a good step direction
using a politician, we use the exact line search (up to machine accuracy)
whenever possible. This eliminates the effect of different line searches
and reduces the number of algorithms we need to test. TFOCS is the only algorithm we use which does not use line search because they do not provide such option.
To compensate on the unfairness to TFOCS, we note that the
algorithm TFOCS uses is accelerated gradient descent and hence we
implement the Gonzaga-Karas's accelerated gradient descent \cite{gonzaga2013fine}, which
is specifically designed to be used with exact line search. Another
reason we pick this variant of accelerated gradient descent is because
we found it to be the fastest variant of accelerated gradient descent (excluding
the geometric descent of \cite{BLS15}) for our tested
data (Gonzaga and Karas also observed that on their own dataset).

The algorithms to be tested are the following:
\begin{itemize}
\item {[}SD{]} Steepest descent algorithm in minFunc.
\item {[}Nes{]} Accelerated gradient descent, General Scheme 2.2.6 in \cite{Nes04}.
\item {[}TFOCS{]} Accelerated gradient descent in TFOCS.
\item {[}GK{]} Gonzaga-Karas's of Accelerated Gradient Descent (Sec \ref{sub:implementation}).
\item {[}Geo{]} Geometric Descent \cite{BLS15}.
\item {[}CG{]} Non-Linear Conjugate Gradient in minFunc.
\item {[}BFGS{]} Broyden\textendash Fletcher\textendash Goldfarb\textendash Shanno
algorithm in minFunc.
\item {[}PCG{]} Preconditioned Non-Linear Conjugate Gradient in minFunc.
\item {[}$\emptyset$+{]} Geometric Politician itself (Sec \ref{sub:implementation}).
\item {[}GK$+${]} Using GK with Geometric Oracle (Sec \ref{sub:implementation}).
\item {[}BFGS$+${]} Using BFGS with Geometric Oracle (Sec \ref{sub:implementation}).
\end{itemize}
We only tested the geometric oracle on GK and BFGS because they are respectively the best
algorithms in theory and practice on our tested data.
The $\emptyset+$ algorithm is used as the control group to test if the
geometric politician by itself is sufficient to achieve good convergence
rate. We note that all algorithms except Nes are parameter free;
each step of SD, Nes, TFOCS, GK, Geo, CG takes $O(n)$ time and each
step of BFGS, PCG, $\emptyset+$, GK$+$ and BFGS$+$ takes roughly $O(nk)$
time for $k^{th}$ iteration.

\subsection{Details of Implementations\label{sub:implementation}}

The first algorithm we implement is the $\emptyset+$ algorithm which
simply repeatedly call the politician. As we will see, this algorithm
is great for non smooth problems but not competitive for smooth problems.

\begin{algorithm2e}
\caption{$\emptyset+$}
\SetAlgoLined

\textbf{Input: }$x_{0}$.

\For{$k\leftarrow 1, 2, \cdots$}{

Set $x_{k+1}\leftarrow\Phi_{f}(x_k, (x_i, f(x_i), \nabla f(x_i))_{i \in [k]})$.

}

\end{algorithm2e}

The second algorithm we implement is the accelerated gradient descent
proposed by Gonzaga and Karas \cite{gonzaga2013fine}. This algorithm uses line search to learn
the the smoothness parameter and strong convexity parameter, see Algorithm \ref{alg:GK}.
We disable the line ({*}) in the algorithm if $\Phi_{f}$ is a politician
instead of an oracle because $\gamma\geq\alpha$ does not hold for
the strong convexity parameter $\alpha$ if $\Phi_{f}$ is not an
oracle. 

\begin{algorithm2e}[t]
\caption{Gonzaga-Karas's variant of Accelerated Gradient Descent} \label{alg:GK}
\SetAlgoLined
\DontPrintSemicolon

\textbf{Input: }$x_{1}$.

$\gamma=2\alpha$, $v_{0}=x_{0}$ and $y_{0}=x_{0}$.

\For{$k\leftarrow 1, 2, \cdots$}{

$y_{k}\leftarrow\Phi_{f}(y_{k-1})$.

$x_{k+1}=\LS(y_{k},-\nabla f(y_{k}))$.

\lIf{$\alpha\geq\gamma/1.02$ and we are using first order oracle}{$\alpha=\gamma/2$. ({*})}

\lIf{$\alpha\geq\frac{\|\nabla f(y_{k})\|^{2}}{2(f(y_{k})-f(x_{k+1}))}$}{$\alpha=\frac{\|\nabla f(y_{k})\|^{2}}{20(f(y_{k})-f(x_{k+1}))}$.}

$G=\gamma\left(\frac{\alpha}{2}\|v_{k}-y_{k}\|^{2}+\left\langle \nabla f(y_{k}),v_{k}-y_{k}\right\rangle \right).$

$A=G+\frac{1}{2}\|\nabla f(y_{k})\|^{2}+(\alpha-\gamma)(f(x_{k})-f(y_{k}))$.

$B=(\alpha-\gamma)(f(x_{k+1})-f(x_{k}))-\gamma(f(y_{k})-f(x_{k}))-G$.

$C=\gamma(f(x_{k+1})-f(x_{k}))$.

$\beta=\frac{-B+\sqrt{B^{2}-4AC}}{2A}$, $\gamma=(1-\beta)\gamma+\beta\alpha$.

$v_{k+1}=\frac{1}{\gamma}((1-\beta)\gamma v_{k}+\beta(\alpha y_{k}-\nabla f(y_{k}))$.

}

\end{algorithm2e}

The third algorithm we implemented is the Broyden\textendash Fletcher\textendash Goldfarb\textendash Shanno
(BFGS) algorithm. This algorithm uses the gradients to reconstruct
the Hessian and use it to approximate Newton's method, see Algorithm \ref{alg:BFGS}. 
We note that another natural way to employ the politician with BFGS is to set $x_{k+1}=\LS(\Phi_{f}(x_{k}),p)$
and this runs faster in practice; however, this algorithm computes two gradients per iteration (namely $\nabla f(x_{k})$ and $\nabla f(\Phi_{f}(x_{k}))$) while we restrict ourselves to algorithms which compute one gradient per iteration.

\begin{algorithm2e}[t]
\caption{BFGS}
\label{alg:BFGS}
\SetAlgoLined
\DontPrintSemicolon

\textbf{Input: }$x_{1}$.

\For{$k\leftarrow 1, 2, \cdots$}{

$p=-\nabla f(x_{k})$.

\For{$i\leftarrow k-1, \cdots, 1$}{

$\alpha_{i}\leftarrow\left\langle s_{i},p\right\rangle /\left\langle s_{i},y_{i}\right\rangle $.

$p=p-\alpha_{i}y_{i}$.

}

$p=\left\langle s_{k-1},y_{k-1}\right\rangle /\left\langle y_{k-1},y_{k-1}\right\rangle p.$

\For{$i\leftarrow 1, \cdots, k-1$}{

$\beta_{i}\leftarrow\left\langle y_{i},p\right\rangle /\left\langle s_{i},y_{i}\right\rangle $.

$p=p+(\alpha_{i}-\beta_{i})y_{i}$.

}

$x_{k+1}=\Phi_{f}\left(\LS(x_{k},p)\right)$.

$s_{k}=x_{k+1}-x_{k}$, $y_{k}=\nabla f(x_{k+1})-\nabla f(x_{k})$.

}

\end{algorithm2e}

\subsection{Quadratic function}

We consider the function
\begin{equation} \label{eq:quad}
f(x)=(x-c)^{\top}D(x-c) ,
\end{equation}
where $D$ is a diagonal matrix with entries uniformly sampled from
$[0,1]$ and $c$ is a random vector with entries uniformly sampled
from the normal distribution $N(0,1)$. Since this is a quadratic
function, CG, BFGS and BFGS+ are equivalent and optimal, namely,
they output the minimum point in the span of all previous gradients.

\begin{figure}[t]
\centering
\includegraphics[width=0.7\textwidth]{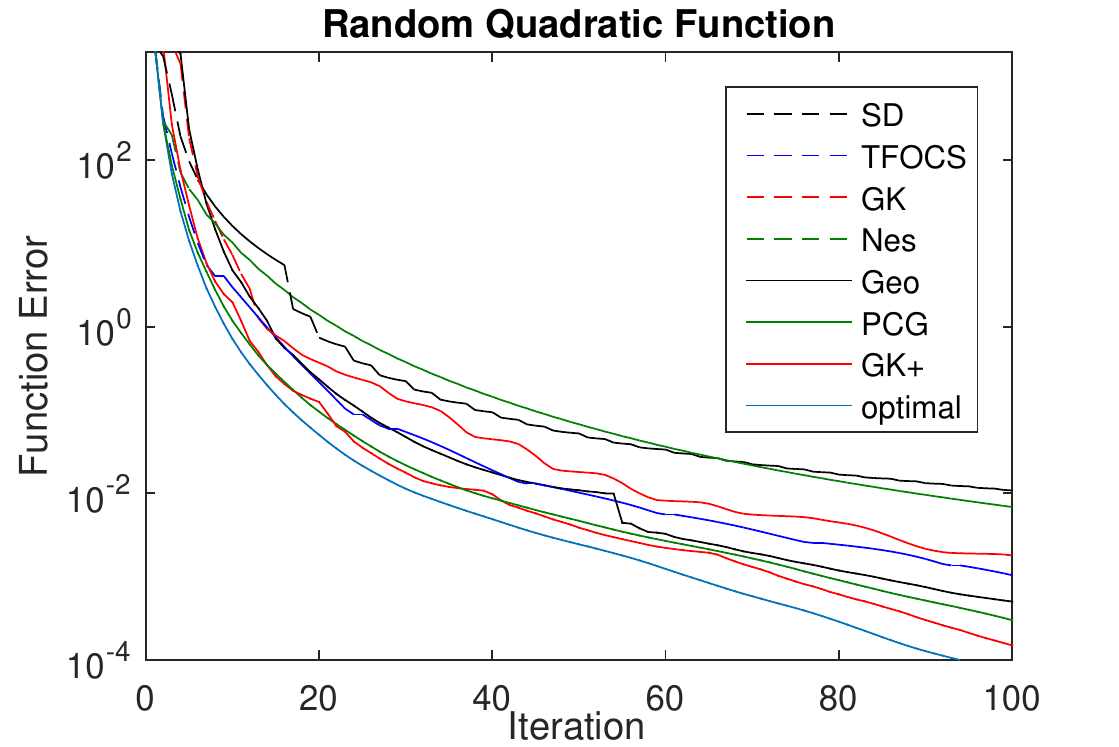}
\caption{Comparison of first-order methods for the function (\ref{eq:quad})
with $n=10000$.}
\end{figure}

\subsection{Variant of Nesterov's Worst Function}

\cite{Nes04} introduced the function
\[
f(x)=(1-x[1])^{2}+\sum_{k=1}^{n-1}(x[k]-x[k+1])^{2}
\]
and used it to give a lower bound for all first-order methods. To
distinguish the performance between CG, BFGS and BFGS$+$, we consider
the following non-quadratic variant
\begin{equation}
f(x)=g(1-x[1])+\sum_{k=1}^{n-1}g(x[k]-x[k+1])\label{eq:variant_nesterov}
\end{equation}
for some function $g$ to be defined. If we pick $g(x)=\left|x\right|$ then
all first order methods takes at least $n$ iterations to minimize $f$ exactly. On the other hand with $g(x)=\max(\left|x\right|-0.1,0)$ one of the minimizer of
$f$ is $(1,\frac{9}{10},\frac{8}{10},\cdots,\frac{1}{10},0,0,\cdots,0)$, and thus it takes at least $11$ iterations for first order methods to minimize $f$ in this case.
We ``regularize'' the situation a bit and consider the function
\[
g(x)=\begin{cases}
\sqrt{\left(x-0.1\right)^{2}+0.001^{2}}-0.001 & \text{if }x\geq0.1\\
\sqrt{\left(x+0.1\right)^{2}+0.001^{2}}-0.001 & \text{if }x\leq-0.1\\
0 & \text{otherwise}
\end{cases}.
\]
Since this function is far from quadratic, our algorithms ($\emptyset+$, GK$+$, BFGS$+$) converge
much faster. This is thus a nice example where the geometric politician helps a lot because the underlying dimension of the problem is small.

\begin{figure}[t]
\centering
\includegraphics[width=0.7\textwidth]{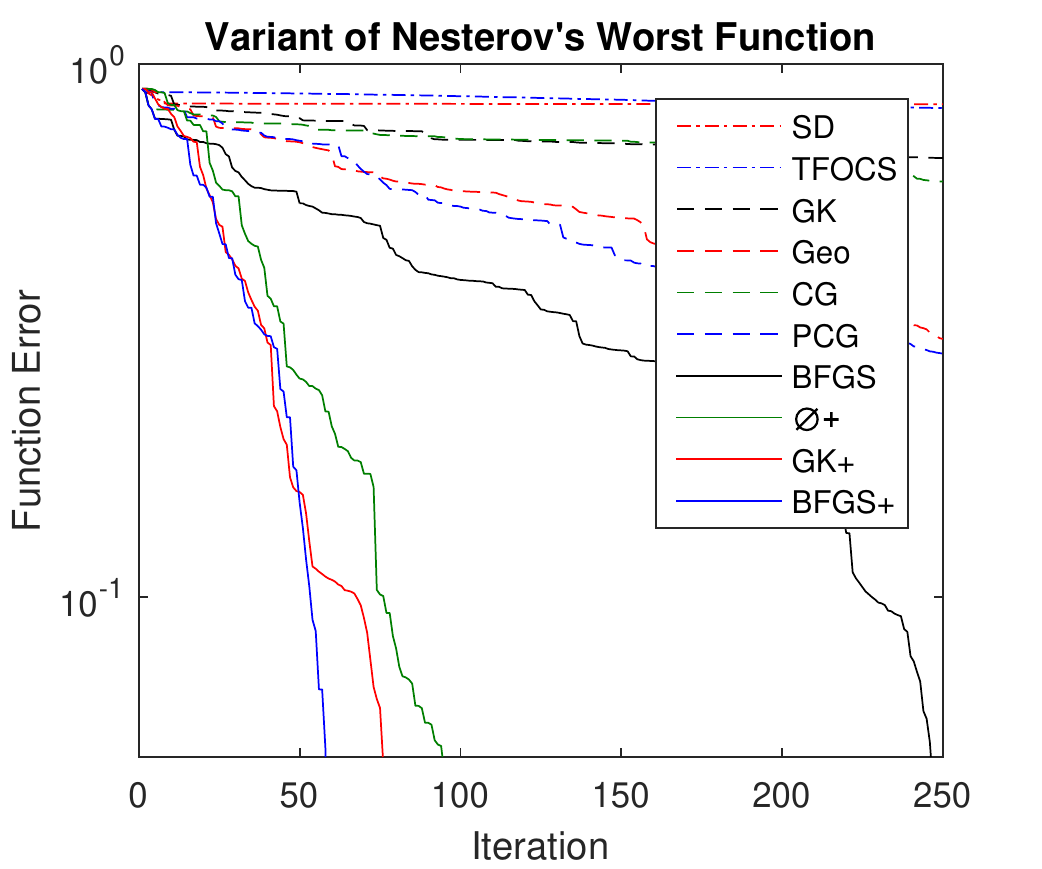}
\caption{Comparison of first-order methods for the function (\ref{eq:variant_nesterov})
with $n=10000$.}
\end{figure}

\subsection{Binary regression with smoothed hinge loss}

We consider the binary classification problem on the datasets from \cite{chang2011libsvm}. The problem is to minimize
the regularized empirical risk: 
\begin{equation}
f_{t}(x)=\frac{1}{n}\sum_{i=1}^{n}\varphi_{t}(b_{i}a_{i}^{T}x)+\frac{\lambda}{2}\left|x\right|^{2}\label{eq:log_regress}
\end{equation}
where $a_{i}\in\R^{d}$, $b_{i}\in\R$ are given by the datasets,
$\lambda$ is the regularization coefficient, $\varphi_{t}$ is the
smoothed hinge loss defined by
\[
\varphi_{t}(z)=\begin{cases}
0 & \text{if }z\leq-1\\
z+1-\frac{t}{2} & \text{if }z\geq-1+t\\
\frac{1}{2t}(z+1)^{2} & \text{otherwise}
\end{cases}
\]
and $t$ is the smoothness parameter. The usual choice for $t$ is
$1$, here we test both $t=1$ and $t=10^{-4}$. The latter
case is to test how well the algorithms perform when the function is
non-smooth. 

We note that for this problem it would be natural to compare ourselves with SGD (stochastic gradient descent) or more refined stochastic algorithms such as SAG \cite{LRSB12} or SVRG \cite{JZ13}. However since the focus of this paper is on general black-box optimization we stick to comparing only to general methods. It is an interesting open problem to extend our algorithms to the stochastic setting, see Section \ref{sec:disc}.

In figures \ref{fig:smooth_profile} and \ref{fig:non_smooth_profile},
we show the performance profile for problems in the LIBSVM datasets (and with different values for the regularization parameter $\lambda$). 
More precisely for a given algorithm we plot $x \in [1,10]$ versus the fraction of datasets that the algorithm can solve (up to a certain prespecified accuracy) in a number of iterations which is at most $x$ times the number of iterations of the best algorithm for this dataset.
Figure \ref{fig:smooth_profile} shows the case $t=1$ with the targeted
accuracy $10^{-6}$; Figure \ref{fig:non_smooth_profile} shows the
case $t=10^{-4}$ with the targeted accuracy $10^{-3}$. We see that TFOCS is slower than SD for many problems, this is simply because
SD uses the line search while TFOCS does not, and this
makes a huge difference for simple problems. Among algorithms taking
$O(n)$ time per iteration, CG and Geo perform the best, while
for the $O(n k)$ algorithms we see that BFGS, BFGS$+$ and GK$+$ perform the best. 
The gap in performance is particularly striking in the non-smooth case where BFGS$+$ is the fastest algorithm on almost all problems and all other methods (except GK$+$) are lagging far behind (for 20\% of the problems all other methods take 10 times more iterations than BFGS$+$ and GK$+$).

Finally in figures \ref{fig:smooth} and \ref{fig:non_smooth} we test
five algorithms on three specific datasets (respectively in the smooth and non-smooth case). In both
figures we see that BFGS$+$ performs the best for all three datasets. BFGS
performs second for smooth problems while GK$+$ performs second for
nonsmooth problems.

\begin{figure}[t]
\centering
\includegraphics[width=0.7\textwidth]{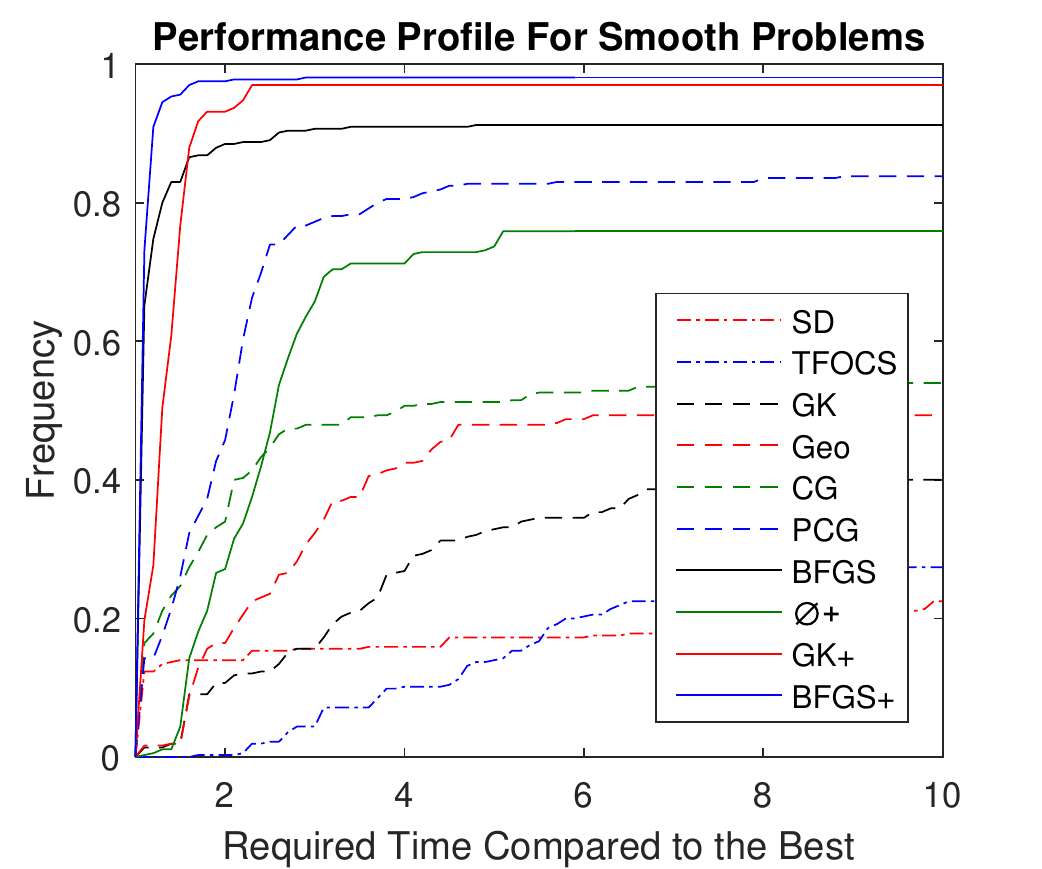}
\caption{Performance profile on problem (\ref{eq:log_regress}) with $t=1$
and $\lambda=10^{-4},10^{-5},10^{-6},10^{-7},10^{-8}$.\label{fig:smooth_profile}}
\end{figure}

\begin{figure}[t]
\centering
\includegraphics[width=0.7\textwidth]{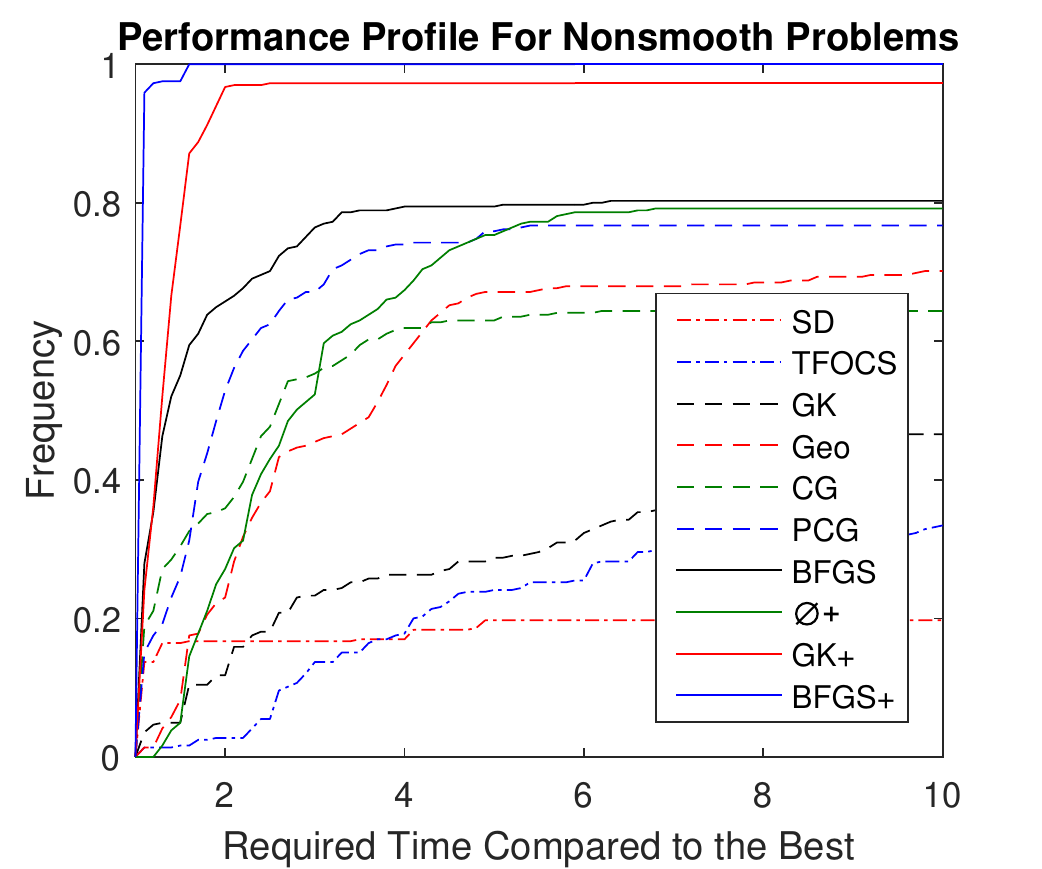}
\caption{Performance profile on problem (\ref{eq:log_regress}) with $t=10^{-4}$
and $\lambda=10^{-4},10^{-5},10^{-6},10^{-7},10^{-8}$.\label{fig:non_smooth_profile}}
\end{figure}

\begin{figure}[t]
\centering
\includegraphics[width=0.7\textwidth]{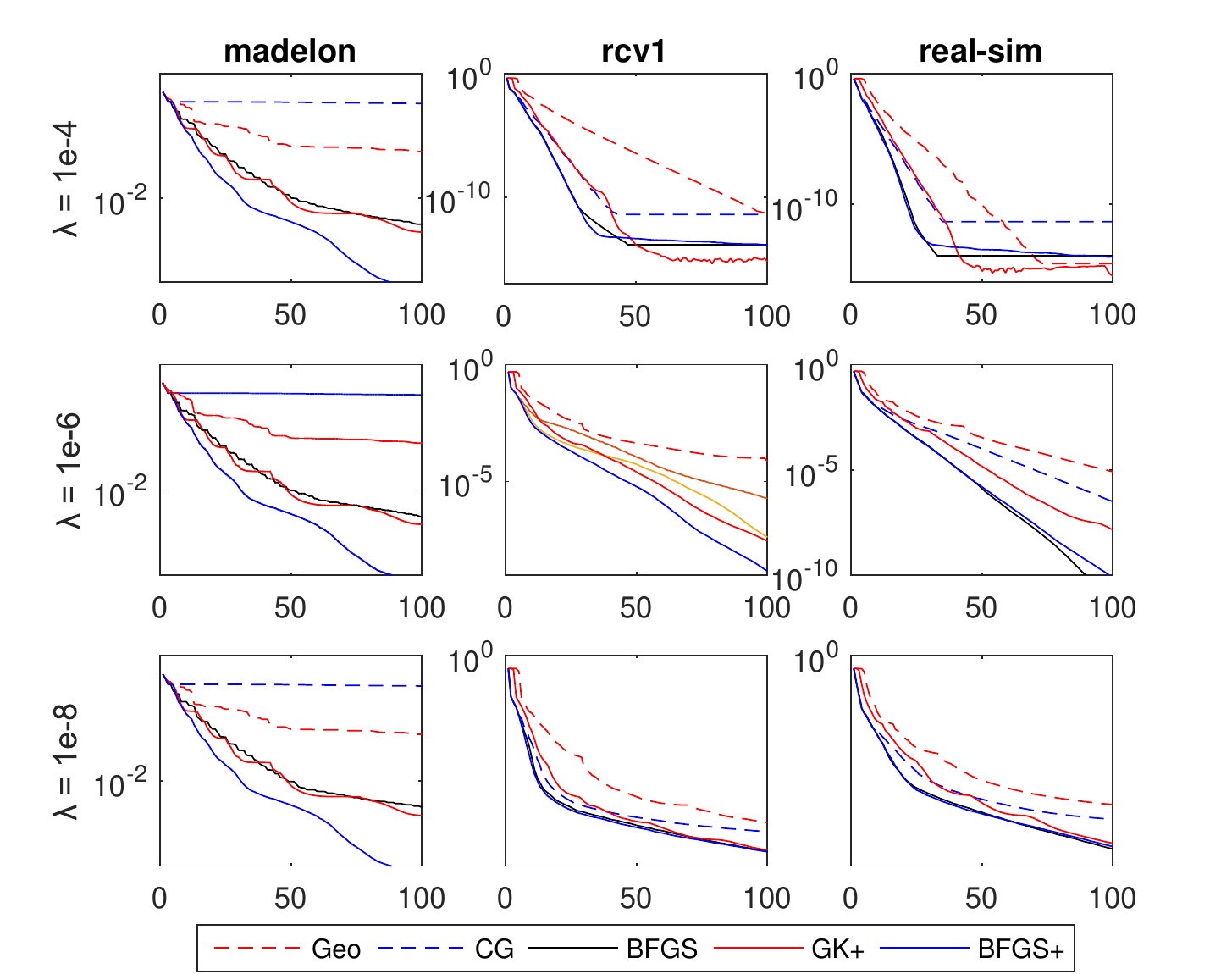}
\caption{Comparison between Geo, CG, BFGS, GK$+$, BFGS$+$ on problem (\ref{eq:log_regress})
with $t=1$ and $\lambda=10^{-4},10^{-6},10^{-8}$.\label{fig:smooth}}
\end{figure}

\begin{figure}[t]
\centering
\includegraphics[width=0.7\textwidth]{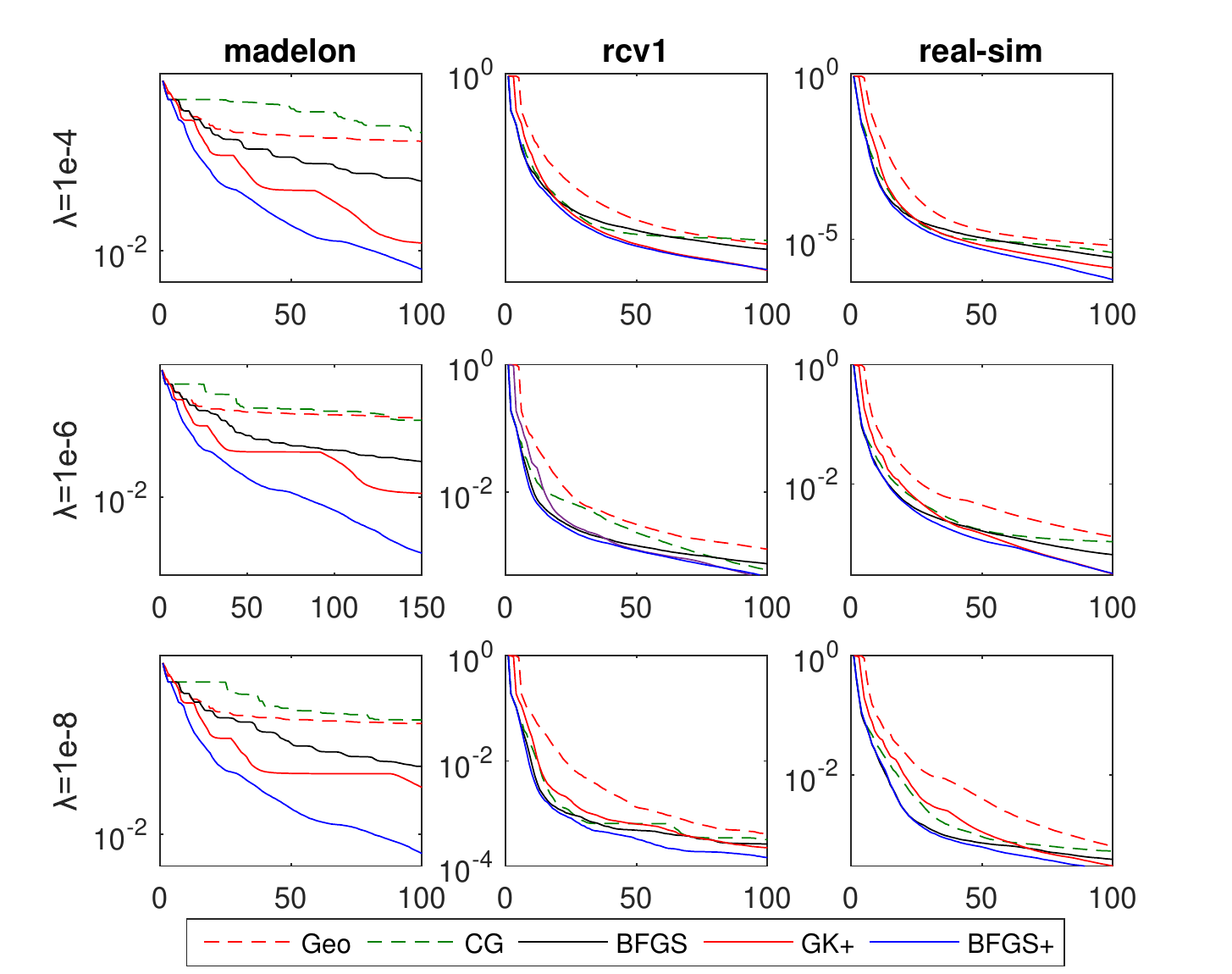}
\caption{Comparison between Geo, CG, BFGS, GK$+$, BFGS$+$ on problem (\ref{eq:log_regress})
with $t=10^{-4}$ and $\lambda=10^{-4},10^{-6},10^{-8}$.\label{fig:non_smooth}}
\end{figure}

\subsection{Summary}
The experiments show that BFGS$+$ and BFGS perform the best among all methods for smooth test problems while BFGS$+$ and GK$+$ perform the best for nonsmooth test problems. The first phenomenon is due to the optimality of these algorithm for quadratic problems. We leave the explanation for the second phenomenon as an open problem. At least, the experiments show that this is not due to the geometric oracle itself since $\emptyset+$ is much slower, and this is not due to the original algorithm since GK performs much worse than GK$+$ for those problems. Overall these experiments are very promising for the geometric oracle as a replacement of quasi Newton method for non-smooth problems and as a general purpose solver due to its robustness.